\documentclass[12pt,letterpaper]{article}
\usepackage{pslatex}
\usepackage{fancyhdr}
\usepackage{graphicx}
\usepackage{geometry}
\RequirePackage[latin1]{inputenc}
 \RequirePackage[T1]{fontenc}

\def\figurename{Figure} % Replace the colon that normally appears after the Figure number by a period.
\makeatletter
\renewcommand{\fnum@figure}[1]{\figurename~\thefigure.}
\makeatother

\def\tablename{Table} % Replace the colon that normally appears after the Figure number by a period.
\makeatletter
\renewcommand{\fnum@table}[1]{\tablename~\thetable.}
\makeatother

\usepackage{amsmath}
\usepackage{amssymb}
\usepackage{amsfonts}
\usepackage{amsthm,amscd}

\newtheorem{theorem}{Theorem}[section]

\newtheorem{proposition}[theorem]{Proposition}
\newtheorem{remark}[theorem]{Remark}
\newtheorem{Definition}[theorem]{Definition}

\numberwithin{equation}{section}

\def\P{\mathbb P}

\def\R{\mathbb R}
\def\E{\mathbb E}

\def\E{\mathbb E}
\def\N{\mathbb N}

\title{$L^{p}$-solutions of backward stochastic differential equations with time-delayed generators}

\author{Yong Ren $^{\,a,}$\footnote{renyong@126.com},\; J. M. Owo$^{\,b,}$\footnote{owo$_{-}$jm@yahoo.fr }\; and\; A. Aman$^{\, b,}$  \footnote{Corresponding author: augusteaman5@yahoo.fr / aman.auguste@ufhb.edu.ci}\\\\
a. Anhui Normal University, Department of Mathematics, Wuhu, Chine\;\;\; \;\;\;\;\;\;\;\;\;\;\;\;\;\;\;\;\;\;\;\;\;\;\;\;\\
b. Université Félix H. Boigny, UFR Mathématiques et Informatique,  Abidjan, C\^{o}te d'Ivoire}
\begin{document}
\maketitle
\date{}
\begin{abstract}
This article is devoted to study the class of backward stochastic differential equation with delayed generator. We suppose the terminal value and the generator to be $L^{p}$-integrable with $p>1$. We derive a new type of estimation related to this BSDE. Next, we establish the existence and uniqueness result in two ways. First, an approximation technics used by Briand et al. (Stochastic Process. Appl. 108 (2003) 109-129) and hence the well-know Picard iterative procedure. Using Picard iterative procedure, we revisit the result of  Dos Reis et al. (Stochastic Process. Appl. 121 (9) (2011) 2114-2150), simplifying the proof and give an explicit existence and uniqueness condition related to the Lipschitz constant $K$ and the terminal time $T$.
\end{abstract}

\vspace{.08in} \noindent {\bf 2000 MR Subject Classification:} 60H15; 60H20; 60H30\\
\vspace{.08in} \noindent \textbf{Keywords:} Backward stochastic differential equation; time delayed generator; $L^{p}$-solutions; $L^{\infty}$-approximation; Picard iterative method.

\section{Introduction}
The aim of this article is concerned with the problem of existence and uniqueness result for the backward stochastic differential equations with delayed generator
\begin{eqnarray}\label{eq1}
Y(t)=\xi +\int_{t}^{T}f(s,Y_{s},Z_{s})ds-\int_{t}^{T}Z(s)dW(s), 0\leq t\leq T,
\end{eqnarray}
where $W$ is a standard Brownian motion. $\xi$ denotes a $\sigma(W_s, 0\leq s\leq T)$-measurable random variable and $f$ called generator depends at each $s$ on $(Y_s,Z_{s})$ defined by $(Y_s,Z_{s})=( Y(s+u),Z(s+u) )_{-T\leq u \leq 0}$. Such equations have been introduced by Delong and Imkeller in \cite{Imk1} and with added Poisson random measures in \cite{Imk2} under delayed Lipschitz condition. Among others, an existence and uniqueness results are been established when the Lipschitz constant $K$ or the terminal time $T$ are small enough. This condition, considered as very restrictive, can unfortunately not be improved. Indeed, in \cite{Imk1}, they exhibited two examples of BSDEs with delayed generator that do not or admit multiple solutions if $K$ the Lipschitz condition or $T$ the terminal time are very large. Two years later, in \cite{Del}, Delong established the link between BSDEs with time delayed generator and the mathematic formulation of many problems in finance and insurance. Since, this kind of BSDE have been intensively studied. Among others one can cite these works of Ren et al., \cite{Ral}, Coulibaly and Aman,\cite{CA}, Tuo et al., \cite{Tal} and Dos Reis et al., \cite{Reis}. All this works have been done under $p$-integrability condition ( $p\geq 2$) on the data $\{f(t,0,0)\}_{t\in[0,T]}$ and $\xi$. In practice, an above condition is too restrictive to be assumed in many applications. For example, let state the following linear BSDE with a delayed generator
\begin{eqnarray}
Y(t)=Y_T +\int_{t}^{T}\int_{-T}^{0}(r(s+u)Y(s+u)+\theta(s+u)Z(s+u))\alpha(du)ds-\int_{t}^{T}Z(s)dW(s),\label{App}	
\end{eqnarray}
which can be considered as the dynamic of an investment portfolio management in insurance or finance. The process $r(t)$ and $\theta(t)$ which represent respectively the interest rate and the risk premium are not necessary bounded. Moreover, when the processes $r$ and $\theta$ are identically null, the solution of BSDE \eqref{App} is a martingale described by $\E(\xi|\mathcal{F}_t)$ that need to be only integrable. In this two previous context, all the above results do not match to study the existence of this model. On other hand, Dos Reis, in \cite{Reis}, established an existence and uniqueness result without more precise condition on time horizon and Lipschitz condition.
   
The aim of this paper is to correct the two previous gaps. Firstly, we briefly revisits the work of Dos Reis giving an explicit condition on the Lipschitz constant and terminal times. Moreover we simplify their proof. When data $\xi$ and $f(t,0,0)$ are  $p$-integrable condition, for any $p\in (1,2)$, we prove, in two different ways, an existence and uniqueness result for BSDEs \eqref{eq1}. To our knowledge, this result does not exist in literature, and hence it is new. On the other hand, the novelty of our two results lies in the fact that we use a new estimation method related to BSDEs with time delayed generator.

 We end this section with the plan of our paper. In Section 2, we give all notations and the setting of the problem. Section 3 contains essential a priori estimates. In Section 4, we prove the existence and uniqueness result.

\section{Notations and Setting of the Problem}
Throughout this paper, we consider the probability space $(\Omega,\mathcal{F},\mathbb{P})$ and $T$ a real and positive constant. We define on the above probability space, $W=(\displaystyle W_{t})_{(0\leq t\leq T)}$ a standard Brownian motion in values in $ \mathbb{R}^{d}$. Let denote $\mathcal{N}$ the class of $\P$-null sets of $\mathcal{F}$ and define
\begin{eqnarray*}
	\mathcal{F}_t=\mathcal{F}^{W}_t\vee\mathcal{N},
\end{eqnarray*}
 where $\mathcal{F}_{t}^{W}=\sigma(W_s,\; 0\leq s\leq T)$. The standard inner product of $\mathbb{R}^{k}$ is denoted by
$\langle \cdot,\cdot\rangle$ and the Euclidean norm by $|\cdot|$.  A
norm on $\mathbb{R}^{d\times k}$ is defined by
$\sqrt{Tr(zz^{\star})}$, where $z^{\star}$ is the transpose of
$z$ will also be denoted by $|\cdot|$.

Next for any real $p>1$ , let consider these following needed spaces.

\begin{description}
\item $\bullet$ Let $L^p(\mathbb{R}^k)$ be the space of $\R^k$-valued and $\mathcal{F}_T$-measurable random variables such that
\begin{eqnarray*}
	\E(|\xi|^p)<+\infty.
\end{eqnarray*}
\item $\bullet$ Let $\mathcal{S}^{p}(\mathbb{R}^{k})$ denote the set of $ \mathbb{R}^{k}$-valued, adapted and càdlàg processes $(\varphi(t))_{t\geq 0}$ such that
\begin{eqnarray*}
\mathbb{E}\left[\left(\sup_{0\leq t\leq T}|\varphi(t)|^{p}\right)\right]< + \infty.
\end{eqnarray*}
\item $\bullet$ Let $\mathcal{H}^{p}(\mathbb{R}^{k \times d})$ denote the set of predictable processes $(\psi(t))_{t\geq 0}$ with value in $\mathbb{R}^{k \times d}$ such that
\begin{eqnarray*}
\mathbb{E}\left[\left(\int_{0}^{T}|\psi(s)|^{2}ds\right)^{p/2}\right]<+\infty.
\end{eqnarray*}
\end{description}
\begin{remark}
We endow spaces $\mathcal{S}^{p}(\mathbb{R}^{k})$ and $\mathcal{H}^{p}(\mathbb{R}^{k \times d})$ respectively with norms $\|.\|_{\mathcal{S}^{p}(\mathbb{R}^{k})}$ and $\|.\|_{\mathcal{H}^{p}(\mathbb{R}^{k\times d})}$ defined by 
\begin{eqnarray*}
\|\varphi\|_{\mathcal{S}^{p}(\mathbb{R}^{k})}=\mathbb{E}\left[\left(\sup_{0\leq t\leq T}|\varphi(t)|^{p}\right)\right]^{1/p}
\end{eqnarray*}
and
\begin{eqnarray*}
\|\psi\|_{\mathcal{H}^{p}(\mathbb{R}^{k\times d})}=\mathbb{E}\left[\left(\int_{0}^{T}|\psi(s)|^{2}ds\right)^{p/2}\right]^{1/p},
\end{eqnarray*}
 that induce on them a Banach space structure.
\end{remark}
We also consider these two additive spaces
\begin{description}
\item $\bullet$ Let $L_{-T}^{2}(\mathbb{R}^{k \times d})$ denote the space of jointly measurable functions         $z:[-T,0]\rightarrow\mathbb{R}^{k \times d}$ such that
\begin{eqnarray*}
\int_{-T}^{0}|z(t)|^{2}dt<+\infty.
\end{eqnarray*}
\item $\bullet$ Let $L_{-T}^{\infty}(\mathbb{R}^{k})$ denote the space of bounded and  jointly measurable functions $y:[-T,0]\rightarrow\mathbb{R}^{k}$ such that
\begin{eqnarray*}
\sup_{-T\leq t\leq 0}|y(t)|^{2}<+\infty.
\end{eqnarray*}
\end{description}
We now recall our backward stochastic differential equation with time delayed generator
\begin{eqnarray}\label{A}
Y(t)=\xi+\int_t^T f(s,Y_s,Z_{s})\,{\rm d}s-\int_t^TZ(s)\,{\rm d}W(s),\quad 0\leq t\leq T,
\end{eqnarray}
where the process $(Y_s,Z_{s}) = (Y(s+u),Z(s+u) )_{-T \leq u\leq 0} $ is at each time $s$, the past values of the solution. In order to extend the solution to interval $[-T,0] $ we always suppose that $Y(s)=Y(0)$ and $Z(s)=0$ for $s<0$.

Now, we make the following assumptions on the data $(\xi,f)$.
\begin{description}
\item [({\bf H1})] $\xi\in L^{p}(\Omega, \mathcal{F};\R^k)$
\item [({\bf H2})] $f:\Omega\times [0,T]\times L_{-T}^{\infty}(\mathbb{R}^{k})\times L_{-T}^{2}(\mathbb{R}^{k\times d})\rightarrow \mathbb{R}^{k}$ be a progressively measurable function such   
\begin{itemize}
\item [$(i)$] $f$ is Lipschitz continuous in the sense that for some probability $\alpha$ defined on $([-T,0],\mathcal{B}([-T,0]))$, there exists a positive constant $K$ such that
\begin{eqnarray*}
\mid f(s,y_s,z_{s}) - f(s,y'_s ,z'_{s} )\mid^2
&\leq &  K\int_{-T}^0 ( \mid y(s+u) - y'(s+u) \mid^2+ \mid z(s+u) - z'(s+u) \mid^2\alpha (du),
\end{eqnarray*}
for $\P\otimes\lambda $-a.e. $(\omega ,s) \in \Omega \times [0,T]$ and for any $(y_s,z_{s}),(y'_s,z'_{s}) \in L_{-T}^{\infty} (\mathbb{R}^k) \times  L_{-T}^2 (\mathbb{R}^{k\times d})$,

\item[$(ii)$] $\displaystyle \mathbb{E}\left[\left(\int_{0}^{T}\left|f(t,0,0)\right|
^{2}dt\right)^{p/2}\right]<+\infty$,
\item[$(iii)$] $\displaystyle f(t,\cdot,\cdot)=0$, for $\ t<0$.
\end{itemize}
\end{description}

To end this section, let us tell what we mean by a solution to BSDE \eqref{A}.

\begin{Definition}
A $L^{p}$-solution of BSDE \eqref{A} is a pair of $\R^{k}\times\R^{k\times d}$-valued process $(Y,Z)$ which satisfies:
\begin{description}
\item[$(i)$] $(Y,Z)\in\mathcal{S}^{p}\times\mathcal{H}^{p}$,
\item[$(ii)$] $(Y,Z)$ satisfies \eqref{A}.
\end{description}
\end{Definition}

\section{A priori Estimates}

\setcounter{theorem}{0} \setcounter{equation}{0}
In this section, for $p>1$, we state a $L^p$-estimates concerning solutions of BSDEs with delayed generator \eqref{A}. We use a new technic without weights spaces and generalize both estimates appear respectively in \cite{Imk1} and \cite{Reis}. Better, these estimates, contrary to one appear in \cite{Reis}, give an explicit condition depending only on data (Lipschitz constant $K$, $T$ and level of integrability $p$) to assure existence and uniqueness result for BSDEs with delayed generator. Indeed, in \cite{Reis}, estimates used by author are more complicated so that authors obtain in Theorem 2.14 this following condition of existence and uniqueness:
\begin{eqnarray}\label{Reis}
2^{p/2-1}C_p\left(\alpha KT\int^0_{-T}e^{-\beta s}\rho(ds)\right)\max(1,T^{p/2})<1.
\end{eqnarray}
where $C_p$ is a constant depending on several parameters which choices seem very complex and require remark 2.11 and Z.12.

We begin with this result which permits us to control the process $Z$ in terms of the data and the process $Y$. 
\begin{proposition}\label{lm1}
Assume $\textbf{(H1)}$-$\textbf{(H2)}$ hold and $T$ and $K$ small enough. Let $(Y,Z)$ be a solution to the delayed BSDE \eqref{A} such that for $p>1$, $Y\in \mathcal{S}^{p}$. Then there exist a positive constant $C_{p}$ depending on
 $p,K,T$ such that
\begin{eqnarray*}
\mathbb{E}\left[\left(\int_{0}^{T}|Z(s)|^{2}ds\right)^{\frac{p}{2}}\right]\leq d_p\E(\sup_{0\leq t\leq T}|Y(t)|^{ p})+C_{p}\mathbb{E}\left[|\xi|^p+\left(\int_{0}^{T}|f(s,0,0)|^{2}ds\right)^{\frac{p}{2}}\right],
\end{eqnarray*}
where $d_p=[2(1-2TK)^{-1}(2K(T+1)+1)]^{p/2}$.
\end{proposition}
\begin{proof}
Let consider a stopping time $(\tau_n)_{n\geq 0}$ defined by 
\begin{eqnarray*}
\tau_{n}= \inf\left\{t\in[0,T], \; \int_{0}^{t}|Z(s)|^{2}ds\geq n \right\}\wedge T.
\end{eqnarray*}
Ito's formula leads to
\begin{eqnarray}
|Y(0)|^{2}+\int_{0}^{\tau_{n}}|Z(s)|^{2}ds&=&|Y(\tau_{n})|^{2}+2\int_{0}^{\tau_{n}}\langle Y(s),f(s,Y_{s},Z_{s})\rangle ds\nonumber \\&&-2\int_{0}^{\tau_{n}}\langle Y(s),Z(s)dW(s) \rangle.\label{N1}
\end{eqnarray}
It follows from assumption (\textbf{H2})
\begin{eqnarray*}
&&2\int_0^{\tau_n}\langle Y(s),f(s,Y_{s},Z_{s})\rangle ds\\
&\leq & 2Y_{*}\int_0^{\tau_n}|f(s,Y_s,Z_s)|ds\\
&\leq & Y_*^2+T\int^{\tau_n}_0|f(s,Y_s,Z_s)|^2ds\\
&\leq & Y_*^2+2TK\int_0^{\tau_n}\left(\int_{-T}^{0}(|Y(s+u)|^2+|Z(s+u)|^2)\alpha(du)\right)ds+2\int^{T}_{0}|f(s,0,0)|^2ds,\\	
\end{eqnarray*}
where $\displaystyle Y_{*}=\sup_{0\leq t\leq T}|Y(t)|$.

Next, applying Fubini's theorem, a change of variables, the fact that $Z(t)=0$ and $Y(t) = Y(0)$ for $t < 0$, it follows that
\begin{eqnarray*}
&&2\int_{0}^{\tau_{n}}\langle Y(s),f(s,Y_{s},Z_{s})\rangle ds\\
&\leq &  Y_{*}^2+ 2TK\int_{-T}^0\left(\int_{u}^{\tau_{n}+u}(|Y(s)|^2+|Z(s)|^2)ds\right)\alpha(du)+2\int_{0}^{T}|f(s,0,0)|^2ds\\
&\leq & (2KT^2+1)\,Y_{*}^2+2TK\int_{0}^{\tau_{n}}|Z(s)|^{2}ds+2\int_{0}^{T}|f(s,0,0)|^{2}ds.
\end{eqnarray*}

Return to \eqref{N1} we have 
\begin{eqnarray*}\label{eq2bis}
\int_{0}^{\tau_{n}}|Z(s)|^{2}ds
&\leq & 2(KT^2+1)Y_{*}^2+2TK\int_{0}^{\tau_{n}}|Z(s)|^{2}ds+2\int_{0}^{T}|f(s,0, 0)|^{2}ds\nonumber\\
&&+\left|\int_{0}^{\tau_{n}}\langle Y(s),Z(s)dW(s)\rangle \right|
\end{eqnarray*}
Since we suppose $2KT<1$, the above inequality becomes
\begin{eqnarray}\label{eq2bis}
\int_{0}^{\tau_{n}}|Z(s)|^{2}ds
&\leq &\frac{2(KT^2+1)}{1-2TK}Y_{*}^2+c\int_{0}^{T}|f(s,0, 0)|^{2}ds+\frac{1}{1-2TK}\left|\int_{0}^{\tau_{n}}\langle Y(s),Z(s)dW(s)\rangle \right|\nonumber\\
\end{eqnarray}

The rest of the proof will be subdivided into two parts of the fact that $p$ is greater or less than $2$.

{\bf First case}: $p\geq 2$\\
Since $p>2$ we have $p/2>1$. Next,  it follows from \eqref{eq2bis} that
\begin{eqnarray}\label{eq2}
\left(\int_{0}^{\tau_{n}}|Z(s)|^{2}ds\right)^{p/2}
&\leq & \left[\frac{2(KT^2+1)}{1-2TK}Y_{*}^2+c\int_{0}^{T}|f(s,0, 0)|^{2}ds+\frac{1}{1-2TK}\left|\int_{0}^{\tau_{n}}\langle Y(s),Z(s)dW(s)\rangle \right|\right]^{p/2}\nonumber\\
&\leq & 2^{p-1}\left(\frac{KT^2+1}{1-2KT}\right)^{p/2}Y_{*}^{p}+c_{p}\left(\int_{0}^{T}|f(s,0, 0)|^{2}ds\right)^{p/2} \nonumber\\
&&+\gamma_p\left|\int_{0}^{\tau_{n}}\langle Y(s),Z(s)dW(s)\rangle \right|^{p/2},
\end{eqnarray}
$\displaystyle \gamma_p=2^{p-2}\left(\frac{1}{1-2TK}\right)^{p/2}$,
where we use the fact that for $a, b, c\in\R$,
\begin{eqnarray}\label{Mink}
|a+b+c|^{p/2}&\leq & 2^{p/2-1}(|a|^{p/2}+|b+c|^{p/2})\nonumber\\
&\leq & 2^{p/2-1}|a|^{p/2}+2^{p-2}|b|+2^{p-2}|c|^{p/2}).
\end{eqnarray}
In virtue of BDG inequality derived in \cite{BDG1} (see, Theorem 3.9.1), there exist a constant $\lambda_p$ defined by
\begin{eqnarray}
\lambda_p=\left(\frac{p}{p-1}\right)^{p^2/2}\left(\frac{p(p-1)}{2}\right)^{p/2}.\label{BDG2}
\end{eqnarray} 
such that
\begin{eqnarray*}
\gamma_p\E\left(\left|\int_{0}^{\tau_{n}}\langle Y(s),Z(s)dW(s)\rangle \right|^{p/2}\right)&\leq & \gamma_p\lambda_p \E\left[\left(\int_{0}^{\tau_{n}}|Y(s)|^2|Z(s)|^2ds \right)^{p/4}\right]\\
&&\leq  \gamma_p\lambda_p \E\left[Y_{*}^{p/2}\left(\int_{0}^{\tau_{n}}|Z(s)|^2ds \right)^{p/4}\right],
\end{eqnarray*}
and thus according to Young inequality, we get 
\begin{eqnarray*}
\gamma_p\E\left(\left|\int_{0}^{\tau_{n}}\langle Y(s),Z(s)dW(s)\rangle \right|^{p/2}\right)&\leq & (\gamma_p\lambda_p)^2\E(Y_{*}^{p})+\frac{1}{2}\E\left[\left(\int_{0}^{\tau_{n}}|Z(s)|^2ds\right)^{p/2}\right].
\end{eqnarray*}
 
Finally according to Fatou's lemma we obtain
\begin{eqnarray}\label{eq2}
\E\left[\left(\int_{0}^{T}|Z(s)|^{2}ds\right)^{p/2}\right]
&\leq & d_pY_{*}^{p}+c_{p}\left(\int_{0}^{T}|f(s,0, 0)|^{2}ds\right)^{p/2},
\end{eqnarray}
where
\begin{eqnarray*}
d_p=2^{p}\left(\frac{KT^2+1}{1-2KT}\right)^{p/2}+2^{2p+1}\left(\frac{p}{p-1}\right)^{p^2}\left(\frac{p(p-1)}{2}\right)^{p}\left(\frac{1}{1-2TK}\right)^{p}
\end{eqnarray*}

{\bf Second case}: $1<p<2$\\
When $p\in(1,2),\; p/2\in (\frac{1}{2}, 1)$ so that we can not use directly the previous method. However, we know that 
\begin{eqnarray}\label{eq21bis}
\left(\int_{0}^{\tau_{n}}|Z(s)|^{2}ds\right)^{1/2}
&\leq &\left[\frac{2(KT^2+1)}{1-2TK}Y_{*}^2+c\int_{0}^{T}|f(s,0, 0)|^{2}ds+\frac{1}{1-2TK}\left|\int_{0}^{\tau_{n}}\langle Y(s),Z(s)dW(s)\rangle \right|\right]^{1/2}\nonumber\\
&\leq &\left(\frac{2(KT^2+1)}{1-2TK}\right)^{1/2}Y_{*}+c\left(\int_{0}^{T}|f(s,0, 0)|^{2}ds\right)^{1/2}\nonumber\\
&&+\left(\frac{1}{1-2TK}\right)^{1/2}\left|\int_{0}^{\tau_{n}}\langle Y(s),Z(s)dW(s)\rangle \right|^{1/2}
\end{eqnarray}

Now raising both sides to the power $p\in(1,2)$ and use again inequality \eqref{Mink}, we get
\begin{eqnarray*}
\left(\int_{0}^{\tau_{n}}|Z(s)|^{2}ds\right)^{p/2}
&\leq &2^{3p/2-1}\left(\frac{KT^2+1}{1-2KT}\right)^{p/2}Y_{*}^{p}+c_{p}\left(\int_{0}^{T}|f(s,0, 0)|^{2}ds\right)^{p/2} \nonumber\\
&&+\bar{\gamma}_p\left|\int_{0}^{\tau_{n}}\langle Y(s),Z(s)dW(s)\rangle \right|^{p/2},
\end{eqnarray*}
where $\displaystyle \bar{\gamma}_p=2^{2p-2}\left(\frac{1}{1-2TK}\right)^{p/2}$.

Recall BDG inequality for $1<p<2$, derived  by Ren in \cite{RF}, there exist a constant $\bar{\lambda}_p$ defined by
\begin{eqnarray}
\bar{\lambda}_p=\left(\frac{4}{p}\right)^{p/4}\frac{4}{4-p},\label{BDG1}
\end{eqnarray}
such that 
\begin{eqnarray*}
\bar{\gamma}_p\E\left(\left|\int_{0}^{\tau_{n}}\langle Y(s),Z(s)dW(s)\rangle \right|^{p/2}\right)&\leq & \bar{\gamma}_p\bar{\lambda}_p \E\left[\left(\int_{0}^{\tau_{n}}|Y(s)|^2|Z(s)|^2ds \right)^{p/4}\right]\\
&\leq & \bar{\gamma}_p\bar{\lambda}_p \E\left[Y_{*}^{p/2}\left(\int_{0}^{\tau_{n}}|Z(s)|^2ds \right)^{p/4}\right]\nonumber\\
&\leq &
\frac{1}{2}(\bar{\gamma}_p\bar{\lambda}_p)^2\E(Y_{*}^{p})+\frac{1}{2}\E\left[\left(\int_{0}^{\tau_{n}}|Z(s)|^2ds\right)^{p/2}\right].
\end{eqnarray*}
Finally, using again Fatou's lemma, we obtain
\begin{eqnarray}\label{eq21}
\E\left[\left(\int_{0}^{T}|Z(s)|^{2}ds\right)^{p/2}\right]
&\leq & d_pY_{*}^{p}+c_{p}\left(\int_{0}^{T}|f(s,0, 0)|^{2}ds\right)^{p/2},
\end{eqnarray}
where
\begin{eqnarray*}
\bar{d}_p=2^{3p/2}\left(\frac{KT^2+1}{1-2KT}\right)^{p/2}+2^{2p+1}\left(\frac{4}{p}\right)^{p/4}\left(\frac{1}{1-2TK}\right)^{p}\frac{1}{4-p}
\end{eqnarray*}
\end{proof}
The second estimate allows us to estimate the solution $(Y,Z)$ in terms of the data. In both cases, this estimation is an essential tool to derive existence and uniqueness result. Since we need the explicit conditions of existence and uniqueness, it seems sensible to no longer use Ito's formula using norms with weights. This compels us in the calculations, to exhibit all the details which from our point of view presents a real difficulty. On the other hand, as in the previous proposition, we have to consider two cases according to the value of $p$.
\begin{proposition}\label{prop1}
Assume $\textbf{(H1)}$-$\textbf{(H2)}$ hold and $K$ and $T$ small enough.  Let $(Y,Z)$ be a solution to the delayed BSDE \eqref{A}. For any $p>1$, if $Y \in \mathcal{S}^{p}$ then, $Z\in \mathcal{H}^{p}$ and there exist
a positive constant $C_{p}$ depending on $p, K, T$ such that
\begin{eqnarray*}
\mathbb{E}\left[\sup_{0\leq t\leq T}|Y(t)|^{ p}+\left(\int_{0}^{T}|Z(t)|^{2}dt\right)^{\frac{p}{2}}\right]\leq C_{p}\mathbb{E}\left[|\xi|^{p}+\left(\int_{0}^{T}|f(t,0,0)|^{2}dt\right)^{\frac{p}{2}}\right].
\end{eqnarray*}
\end{proposition}
\begin{proof}
In view of \eqref{A}, it is easy to check that
\begin{equation*}
Y(t)\leq |\xi| + \int_{t}^{T}|f(s,Y_{s},Z_{s})|ds-\int_{t}^{T}Z(s)dW(s),\;\; 0\leq t\leq T.
\end{equation*}
The conditional expectation with respect to $\mathcal{F}_{t}$ taken in the previous inequality provides
\begin{equation*}
Y(t) \leq \mathbb{E}\left( |\xi| + \int_{0}^{T}|f(s,Y_{s},Z_{s})|ds \Big|\mathcal{F}_{t} \right), \;\; 0\leq t\leq T.
\end{equation*}
Moreover, applying respectively Doob's and Jensen's inequalities, we obtain
\begin{eqnarray}\label{IN3}
\mathbb{E}\left(\sup_{0\leq t\leq T}|Y(t)|^{p}\right) &\leq & \mathbb{E}\left(\sup_{0\leq t\leq T }\left|\mathbb{E}\left(|\xi|+\int_{0}^{T}|f(s,Y_{s},Z_{s})|ds \Big|\mathcal{F}_{t} \right)\right|^{p}\right)\nonumber\\
&\leq & \left(\frac{p}{p-1}\right)^{p}\sup_{0\leq t\leq T}\mathbb{E}\left[\left|\mathbb{E}\left(|\xi|+ \int_{0}^{T}|f(s,Y_{s},Z_{s})|ds\Big|\mathcal{F}_{t}\right)\right|^{p}\right]\nonumber\\
&\leq &2^{p-1}\left(\frac{p}{p-1}\right)^{p}\mathbb{E}\left[|\xi|^p + T^{p/2}\left(\int_{0}^{T}|f(s,Y_{s},Z_{s})|^2ds\right)^{p/2}\right].
\end{eqnarray}
Let now estimate the second term of the right side of \eqref{IN3}. For this purpose, we have
\begin{eqnarray}\label{eq4}
\int_{0}^{T}|f(s,Y_{s},Z_{s})|^2ds
&\leq & 2\int_{0}^{T}|f(s,Y_{s},Z_{s})-f(s,0,0)|^2ds +2\int_{0}^{T}|f(s,0,0)|^2ds\nonumber\\
&\leq & 2K\int_{0}^{T}\int_{-T}^{0}|Y(s+u)|^2\alpha(du)ds+2K\int_{0}^{T}\int_{-T}^{0}|Z(s+u)|^2\alpha(du)ds\nonumber\\
&&+2\int_{0}^{T}|f(s,0,0)|^2ds\nonumber\\
&\leq & 2KTY_{*}^2+2K\int_0^T|Z(s)|^2ds+2\int_{0}^{T}|f(s,0,0)|^2ds,
\end{eqnarray}
where we use Lipschitz assumption, Fubini inequality, change of variable and the fact $Z(t)=0$ for $t < 0$.
The rest of this proof is subdivided in two step.

{\bf First step}: $p\geq 2$\\
We have
\begin{eqnarray}\label{eq
4}
T^{p/2}\left(\int_{0}^{T}|f(s,Y_{s},Z_{s})|^2ds\right)^{p/2}
&\leq &  2^{p-1}K^{p/2}T^pY_{*}^{p}+2^{3p/2-2}K^{p/2}T^{p/2}\left(\int_0^T|Z(s)|^2ds\right)^{p/2}\nonumber\\
&&+c_p\left(\int_{0}^{T}|f(s,0,0)|^2ds\right)^{p/2}
\end{eqnarray}
which together with Proposition 3.1 and inequality \eqref{IN3} provide
\begin{eqnarray*}\label{IN4}
\mathbb{E}\left(\sup_{0\leq t\leq T}|Y(t)|^{p}\right)
&\leq & 2^{2p-2}\left(\frac{p}{p-1}\right)^p K^{p/2}\left(T^{p/2}+2^{p/2-1}K^{p/2}d_p\right)\mathbb{E}\left(\sup_{0\leq t\leq T}|Y(t)|^{p}\right)\nonumber\\
&&+c_p\E\left[|\xi|^{p}+\left(\int_{0}^{T}|f(s,0,0)|^{2}ds\right)^{p/2}\right].
\end{eqnarray*}
Finally, choosing $T$ and $K$ such that $\beta_p<1$, where $\displaystyle \beta_p=2^{2p-2}\left(\frac{p}{p-1}\right)^p K^{p/2}\left(T^{p/2}+2^{p/2-1}K^{p/2}d_p\right)$, we get 
\begin{eqnarray}
\mathbb{E}\left(\sup_{0\leq s \leq T}|Y(s)|^{p}\right)&\leq & C_p\mathbb{E}\left[|\xi|^{p}+\left(\int_{0}^{T}|f(s,0,0)|^{2}ds\right)^{p/2}\right].
\end{eqnarray}

{\bf Second step}: $1<p<2$\\
We have
\begin{eqnarray}\label{eq4}
T^{p/2}\left(\int_{0}^{T}|f(s,Y_{s},Z_{s})|^2ds\right)^{p/2}
&\leq &  2^{3p/2-1}K^{p/2}T^pY_{*}^{p}+2^{5p/2-2}K^{p/2}T^{p/2}\left(\int_0^T|Z(s)|^2ds\right)^{p/2}\nonumber\\
&&+c_p\left(\int_{0}^{T}|f(s,0,0)|^2ds\right)^{p/2}
\end{eqnarray}
which together with Proposition 3.1 and inequality \eqref{IN3} provide
\begin{eqnarray*}\label{IN4}
\mathbb{E}\left(\sup_{0\leq t\leq T}|Y(t)|^{p}\right)
&\leq & 2^{5p/2-2}\left(\frac{p}{p-1}\right)^p K^{p/2}\left(T^{p/2}+2^{p-1}K^{p/2}\bar{d}_p\right)\mathbb{E}\left(\sup_{0\leq t\leq T}|Y(t)|^{p}\right)\nonumber\\
&&+c_p\E\left[|\xi|^{p}+\left(\int_{0}^{T}|f(s,0,0)|^{2}ds\right)^{p/2}\right].
\end{eqnarray*}
Finally, choosing $T$ and $K$ such that $\bar{\beta}_p<1$, where $\displaystyle \bar{\beta}_p=2^{5p/2-2}\left(\frac{p}{p-1}\right)^p K^{p/2}\left(T^{p/2}+2^{p-1}K^{p/2}\bar{d}_p\right)$, we get 
\begin{eqnarray}
\mathbb{E}\left(\sup_{0\leq s \leq T}|Y(s)|^{p}\right)&\leq & C_p\mathbb{E}\left[|\xi|^{p}+\left(\int_{0}^{T}|f(s,0,0)|^{2}ds\right)^{p/2}\right].
\end{eqnarray}
\end{proof}

\begin{remark}
According to Proposition \ref{lm1} and \ref{prop1}, it is reasonable to clarify the terminology "$K$ and $T$ small enough". Indeed, in view of their proofs, the terminal time $T$ and the Lipschitz constant $K$ must verify this following conditions according to the values of $p$.
\begin{eqnarray}
2KT<1,\label{Cond2}
\end{eqnarray}

\begin{eqnarray}\label{Cond3}
2^{2p-2}\left(\frac{p}{p-1}\right)^p K^{p/2}\left(T^{p/2}+2^{p/2-1}K^{p/2}d_p\right)<1,\;\;\;\mbox{for}\;\; p\geq 2,
\end{eqnarray}
and 
\begin{eqnarray}\label{Cond3bis}
2^{5p/2-2}\left(\frac{p}{p-1}\right)^p K^{p/2}\left(T^{p/2}+2^{p-1}K^{p/2}\bar{d}_p\right)<1,\;\;\;\mbox{for}\;\; 0< p< 2. 	
\end{eqnarray}
\end{remark}
\section{Existence and uniqueness}
This section is devoted to establish the existence and uniqueness result for BSDE \eqref{A} in $L^p$-sense, for $p>1$.

When $p>2$, we revisit the result of Reis et al. established in \cite{Reis} by clarifying the condition of existence and uniqueness. Indeed, since the data (terminal value $\xi$ and the process $f(s,0,0)$) are $p$-integrable for $p>2 $, one know that they are also square integrable. Therefore it suffices to prove by Picard's iteration principle existence and uniqueness result in $L^2$ with a similar method used by Delong and Imkeller in \cite{Imk1}. Then show with Proposition 3.1 and 3.2 that this solution lives in $L^p$.

Whereas for $p\in (1,2)$, data (terminal value $\xi$ and the process $f(s,0,0)$) are in $L^p, \; 1<p<2$ so that they are not square integrable. Therefore we can not use $L^{2}$-existence and uniqueness result due by Delong and Imkeller in \cite{Imk1}. To get around this difficulty, we will use two methods. First, we use the technic inspired by Briand et al. in \cite{Bral} that is the $L^{\infty}$-approximation which consist to build a sequence and study it convergence in $L^p$. Next, the second method is the well know Picard iterative procedure combining with $L^p$-solution derive by Briand et al. (see \cite{Bral}) for non delayed BSDE.  

Let first recall the $L^2$-existence and uniqueness with a slight modification on the condition due to our new estimation method and choice of the space of solution.

\begin{theorem}\label{th1}
Assume \textbf{(H1)}-\textbf{(H2)} and  $T$ and $K$ satisfy
\begin{eqnarray}\label{th10}
28TK\max\{1,T\}<1.
\end{eqnarray}
Then BSDE \eqref{A} has as a unique solution in
$\mathcal{S}^{2}\times \mathcal{H}^{2}$.
\end{theorem}
Next we give $L^p$-solution for $p>2$.
\begin{theorem}
Suppose that $p>2$ and assume that ({\bf H1})-({\bf H2}) hold. Let $K$ or $T$ be small enough such that \eqref{Cond2} and \eqref{Cond3} hold. Then BSDE \eqref{A} admits a unique solution $(Y,Z)\in\mathcal{S}^{p}\times \mathcal{H}^{p}$.
\end{theorem}

\begin{proof}
Since the proof is based on the standard Picard iteration, we initialize by $Y^0 = 0$ and $Z^0 = 0$ and define recursively	, for $n\geq 0$,
\begin{eqnarray}
Y^{n+1}(t)=\xi+\int^T_t f(s,Y_s^{n},Z_s^{n})ds-\int^T_t Z^{n+1}(s)dW_s,\;\; 0\leq t\leq T.\label{Ind}
\end{eqnarray}
For $n\geq 1$, let assume $(Y^n,Z^n)\in\mathcal{S}^{p}\times \mathcal{H}^{p}$. We shall prove that \eqref{Ind} has a unique solution $(Y^{n+1},Z^{n+1})\in\mathcal{S}^{p}\times \mathcal{H}^{p}$. In this fact, in view of \eqref{eq4}, we derive
\begin{eqnarray*}
\left(\int^{T}_{0}|f(t,Y^n_s,Z^n_s)|^2dt\right)^{p/2} &\leq & 2^{p} 2^{p-1}K^{p/2}T^{p/2}(Y_{*}^{n})^{p}+c_p\left(\int_{0}^{T}|f(s,0,0)|^{2}ds\right)^{p/2}\nonumber\\
&&+2^{3p/2-2}K^{p/2}\left(\int_{0}^{T}|Z^n(s)|^{2}ds\right)^{p/2}	\\
&<& +\infty,
\end{eqnarray*}
 which prove that the process $f(s,Y^n_s,Z^n_s)$ belongs in $\mathcal{H}^{p},\ p>2$, hence in $\mathcal{H}^{2}$. Next, $\xi$ belongs also in $L^2$ because it is state in $L^p, \ p>2$. Therefore, the martingale representation yields a unique process $Z^{n+1}\in\mathcal{H}^2$ such that for $t\in[0,T]$
\begin{eqnarray*}
\E\left(\xi+\int_0^Tf(s,Y^n_s,Z^n_s)ds|\mathcal{F}_t\right)=\E\left(\xi+\int_0^Tf(s,Y^n_s,Z^n_s)ds\right)+\int_0^t Z^{n+1}dW(s).
\end{eqnarray*}  
Setting $\displaystyle Y^{n+1}_t=\E\left(\xi+\int_t^Tf(s,Y^n_s,Z^n_s)ds| \mathcal{F}_t\right)$, then $(Y^{n+1},Z^{n+1})$ is 	a unique solution of BSDEs \eqref{Ind}. Moreover it follows from Proposition \ref{lm1} and Proposition \ref{prop1} that $(Y^{n+1},Z^{n+1})\in \mathcal{S}^p\times\mathcal{H}^p$ and hence in $\mathcal{S}^2\times\mathcal{H}^2$. 

 A similar approach used by Delong and Imkeller in the proof of their result (see Theorem 2.1 in \cite{Imk1}) implies that $(Y^n,Z^n)$ converges in $\mathcal{S}^2\times\mathcal{H}^2$. Therefore it limit denoted by $(Y,Z)$ satisfies BSDE \eqref{A}. It remains to prove that $(Y,Z)$ belongs to $\mathcal{S}^p\times\mathcal{H}^p$. This is confirmed using once again respectively Proposition \ref{lm1} and Proposition \ref{prop1}.
 \end{proof}

We are now ready to give the main result of this paper which is the existence and uniqueness result for BSDEs \eqref{A} under $p$-integrable data, with $p\in(1,2)$.
\begin{theorem}
Assume \textbf{(H1)}-\textbf{(H2)} and $T$ or $K$ small enough i.e $T$ and $K$. Then BSDE  \eqref{A} has as a unique solution in $\mathcal{S}^{p}\times\mathcal{H}^{p}$, for  $p\in(1,2)$.
\end{theorem}
\begin{proof}
Let us start by studying the uniqueness part. Let us consider $(Y,Z)$ and $(Y',Z')$ two solutions of BSDE \eqref{A} in the appropriate space. Then the process $(U,V)$ defined by $U=Y-Y'$ and $V=Z-Z'$ is a solution of the following BSDE
\begin{eqnarray*}
	U(t)=\int_t^Tg(s,U_s,V_s)ds-\int_t^TV(s)dW(s),
\end{eqnarray*}
where $g$ stands for the random function
\begin{eqnarray*}
g(t,y_t,z_t) = f(t,y_t+Y'_t,z_t+Z'_t)-f(t,Y'_T,Z'_t)
\end{eqnarray*}
It is not difficult to prove that $g$ satisfies assumption \textbf{(H2)}, with $g(t,0,0)=0$. Thanks to Proposition \ref{prop1}, we get immediately that $(U,V)=(0,0)$ and finally $Y=Y'$ and $Z=Z'$.

Let us turn to the existence part. We will use two methods.

{\bf $L^{\infty}$-approximation method}\\

Let $(q_n)_{n\geq 1}$ denote the  sequence of function defined by
\begin{eqnarray*}
q_{n}(x)= x\frac{n}{|x|\vee n}, \; \forall\;\; x\in \R. 
\end{eqnarray*}
Next we define
\begin{eqnarray*}
\xi_n= q_{n}(\xi),\;\;\;\; f_n(t,y,z)-f(t,0,0)+q_n(f(t,0,0)). 
\end{eqnarray*}
Since for each $n\geq 1$, $|q_{n}(x)|\leq n$, $\xi_n$ and $f_n(t,0,0)$ are bounded random variable and hence are square integrable. Therefore, according to Theorem \ref{th1}, there exists a unique process $(Y^n,Z^n)\in \mathcal{S}_0^2\times \mathcal{H}_0^2$ such that
\begin{eqnarray}
	Y^n(t)=\xi_n+\int_t^T f_n(s,Y^n_s,Z^n_s)ds-\int_t^TZ^n(s)dW(s), \;\; 0\leq t\leq T.\label{Fn}
\end{eqnarray}
Moreover thanks to Proposition \ref{prop1}, $(Y^{n},Z^{n})$ is also in $\mathcal{S}^{p}\times \mathcal{H}^{p}$, for all $n\geq 1$. Applying again Proposition \ref{prop1}, for $(n,i)\in\mathbb{N}^{\ast}\times\mathbb{N}$, we have
\begin{eqnarray*}
&&\mathbb{E}\left[\sup_{0\leq t\leq T}|Y^{n+i}(t)-Y^{n}(t)|^{p}+\left(\int_{0}^{T}|Z^{n+i}(s)-Z^{n}(s)|^{2}ds\right)^{\frac{p}{2}}\right]
\\ &\leq&\bar{C}_{p}\mathbb{E}\left[|\xi_{n+i}-\xi_{n}|^{p}+\left(\int_{0}^{T}|q_{n+i}(f(t,0,0))-q_{n}(f(t,0,0))|^{2}dt\right)^{\frac{p}{2}}\right],
\end{eqnarray*}
where $\bar{C}_p$ depends only $p,\; T$ and $K$.
It is clear that the right hand-side of the last inequality tends to $0$ as $n$ tends to $+\infty$, uniformly in $i\in\N$. Therefore $(Y^{n},Z^{n})$ is a Cauchy sequence in Banach space $\mathcal{S}^{p}\times \mathcal{H}^{p}$ and converges to a process $(Y, Z)\in \mathcal{S}^{p}\times \mathcal{H}^{p}$. Finally taking the limit in  \eqref{Fn}, we derive that $(Y,Z)$ satisfies \eqref{A}. All this implies that the BSDE \eqref{A} admit a solution.

{\bf Picard iterative procedure}\\
To use the this method let us recall this needed result which follows from Theorem 4.1 in \cite{Bral}. For this let suppose $f$ do not depend to the state $y_t$ and $z_t$ and consider the BSDE
\begin{eqnarray}\label{Z1}
	Y(t)=\xi+\int^T_t f(s)ds-\int^T_t Z(s)dW(s).
\end{eqnarray} 
\begin{theorem}\label{TQ}
Assume \textbf{(H1)} and $f$ is $p$-integrable. Then BSDE  \eqref{TQ} has as a unique solution in $\mathcal{S}^{p}\times\mathcal{H}^{p}$, for  $p\in(1,2)$.
\end{theorem}
For the general $f$, let $Y^0(t)=Z^0(t)=0$ and define recursively for $n\geq 1$
\begin{eqnarray}\label{Z2}
Y^{n}(t)=\xi+\int^T_t f(s,Y^{n-1}_s,Z^{n-1}_s)ds-\int^T_t Z^{n}(s)dW(s).
\end{eqnarray}
Based on Lipschitz assumption of $f$, we derive easily that a process $g$ defined by $g(t)= f(t, Y^{n-1}_t,Z^{n-1}_t)$, for a given $(Y^{n-1},Z^{n-1})\in\mathcal{S}^p\times\mathcal{H}^{p}$, belongs to $\mathcal{H}^p$. Next, since $\xi$ is $p$-integrable, in view of Theorem \ref{TQ}, BSDE associated with data $\xi$ and $g$ admit a unique solution denoted by $(Y^n,Z^n)$.
Actually let us show that the sequence of processus $(Y^n, Z^n)$ converge. In this fact, setting $\bar{Y}^{n}=Y^{n}-Y^{n-1}$ and  $\bar{Z}^{n}=Z^{n}-Z^{n-1}$, using respectively Doob inequality,  Lipschitz assumption, Fubini's equality and $Z^{n}(t)=0$ if $t<0$ for all $n\geq 1$, we derive
\begin{eqnarray}\label{RT1}
&&\E\left(\sup_{0\leq t\leq T}|\bar{Y}^{n}(t)|^{p}\right)\nonumber\\
&\leq & \left(\frac{p}{p-1}\right)^pT^{p/2}\E\left[\left(\int_{0}^{T}|f(s,Y^{n-1}_s,Z^{n-1}_s)-f(s,Y^{n-2}_s,Z^{n-2})|^2ds\right)^{p/2}\right]\nonumber\\
&\leq & \left(\frac{p}{p-1}\right)^p T^{p/2} K\E\left[\left(\int_0^T\int_{-T}^{0}(|\bar{Y}^{n-1}(s+u)|^2+|\bar{Z}^{n-1}(s+u)|^2)\alpha(du)ds\right)^{p/2}\right]\nonumber\\
&\leq & \left(\frac{p}{p-1}\right)^p T^{p/2}K\E\left[\left(\int_{-T}^{0}\int_{u}^{T+u}(|\bar{Y}^{n-1}(s)|^2+|\bar{Z}^{n-1}(s)|^2)ds\alpha(du)\right)^{p/2}\right]\\
&\leq & \left(\frac{p}{p-1}\right)^p2^{p/2-1}K^{p/2}T^{p/2}\max(T^{p/2},1)\E\left(\sup_{0\leq t \leq T}|\bar{Y}^{n-1}(t)|^p+\left(\int_0^T|\bar{Z}^{n-1}(s)|^2ds\right)^{p/2}\right).\nonumber
\end{eqnarray}
 On the other hand, with the same computation as in the proof of Proposition 3.1, we obtain
 \begin{eqnarray}\label{Ano1}
 \E\left[\left(\int^{T}_0|\bar{Z}^n(s)|^{2}ds\right)^{p/2}\right] &\leq &
 2^{2p-2}\E[(\bar{Y}^n_{*})^p]+2^{2p-2}T^pK^{p/2}\E[(\bar{Y}^{n-1}_{*})^p]\nonumber\\
 &&+2^{2p-2}T^{p/2}K^{p/2}\E\left(\int^{T}_0|\bar{Z}^{n-1}(s)|^{2}ds\right)^{p/2}\nonumber\\
 &&+2^{2p-2}\E\left(\left|\int^{T}_0\bar{Y}^{n}(s)\bar{Z}^{n}(s)dW(s)\right|^{p/2}\right).
 \end{eqnarray}
 But with BDG inequality we have
 \begin{eqnarray*}
\E\left(\left|\int^{T}_0\bar{Y}^{n}\bar{Z}^{n}(s)^{2}dW(s)\right|^{p/2}\right)&\leq &\bar{\lambda}_p \E\left(\left(\int^{T}_0|\bar{Y}^{n}(s)|^2|\bar{Z}^{n}(s)|^{2}ds\right)^{p/4}\right)\nonumber\\
&\leq &\frac{1}{2}\bar{\lambda}_p^2\E[(\bar{Y}^n_{*})^p]+\frac{1}{2}\E\left[\left(\int^{T}_0|\bar{Z}^n(s)|^{2}ds\right)^{p/2}\right]
 \end{eqnarray*}
 which together with \eqref{Ano1} provide
 \begin{eqnarray}\label{Ano2}
 &&\E\left[\left(\int^{T}_0|\bar{Z}^n(s)|^{2}ds\right)^{p/2}\right]\nonumber\\
  &\leq &
 2^{2p-2}(2+\bar{\lambda}_p^2)\E[(\bar{Y}^n_{*})^p]\nonumber\\ &&+2^{2p-1}T^{p/2}K^{p/2}\max(1,T^{p/2})\E\left[(\bar{Y}^{n-1}_{*})^p+\left(\int^{T}_0|\bar{Z}^{n-1}(s)|^{2}ds\right)^{p/2}\right].
 \end{eqnarray}
 Finally, combining \eqref{Ano2} and \eqref{RT1} we get
 \begin{eqnarray*}
 \E\left[\sup_{0\leq t\leq T}|\bar{Y}^{n}(t)|^{p}+\left(\int^{T}_0|\bar{Z}^n(s)|^{2}ds\right)^{p/2}\right]\leq C_p(K,T)\E\left[\sup_{0\leq t\leq T}|\bar{Y}^{n-1}(t)|^{p}+\left(\int^{T}_0|\bar{Z}^{n-1}(s)|^{2}ds\right)^{p/2}\right],
 \end{eqnarray*}
 where 
 \begin{eqnarray}\label{123}
 	C_p(K,T)=2^{2p-1}T^{p/2}K^{p/2}\max(1,T^{p/2})\left[1+2^{p/2-2}(2+\bar{\lambda}_p^2)\left(\frac{p}{p-1}\right)^p\right]\nonumber\\
\end{eqnarray}
 By iteration, we end up with
  \begin{eqnarray}\label{RT3}
 \E\left[\sup_{0\leq t\leq T}|\bar{Y}^{n}(t)|^{p}+\left(\int^{T}_0|\bar{Z}^n(s)|^{2}ds\right)^{p/2}\right]\leq (C_p(K,T))^{n-1}\E\left[\sup_{0\leq t\leq T}|\bar{Y}^{1}(t)|^{p}+\left(\int^{T}_0|\bar{Z}^{1}(s)|^{2}ds\right)^{p/2}\right].\nonumber\\
 \end{eqnarray}
 Since we assume that $K$ and $T$ are small enough such that $C_p(K,T)<1$, \eqref{RT3} yields that $(Y^n,Z^n)_{n\geq 1}$ defined by \eqref{Z2} is a Cauchy sequence on the space $\mathcal{S}^p\times \mathcal{H}^p$, a Banach space. Therefore there exists a $(Y,Z)\in \mathcal{S}^p\times \mathcal{H}^p $ limit of the converging sequence $(Y^n,Z^n)_{n\geq 1}$, solution of BSDE \eqref{A}.  
\end{proof}
\begin{remark}
\begin{itemize}
\item [$(i)$] In Picard's iterative method, the terminology "$K$ and $T$ small enough" results in la condition \eqref{123}. While in the $L^{\infty}$-approximation method the terminology "$K$ and $T$ small enough" satisfies \eqref{Cond2} and \eqref{Cond3}.
\item [$(ii)$] It is also useful to note that the conditions \ \eqref{Cond2} and \eqref{Cond3bis} are sufficient conditions for the existence and uniqueness of BSDEs \ eqref {A}. Whereas the condition \eqref{123} is a necessary and sufficient condition of existence and uniqueness of this same BSDE.
\end{itemize}
At the end of this study, we can affirm that the condition of existence on the Lipschitz constant and the terminal time depends on the method used.
\end{remark}

\label{lastpage-01}
\end{document}